\documentclass[conference]{IEEEtran}
\IEEEoverridecommandlockouts
\usepackage{cite}
\usepackage{amsmath,amssymb,amsfonts,url,bm,enumerate}
\usepackage{amsthm}
\usepackage{algorithmic}
\usepackage{graphicx}
\usepackage{mathtools}
\usepackage{textcomp}
\usepackage{xcolor}
\usepackage{tikz}
\usepackage{float}
\usepackage{ctable}
\usepackage{booktabs}
\usepackage{array}
\usepackage{multicol}
\def\BibTeX{{\rm B\kern-.05em{\sc i\kern-.025em b}\kern-.08em
    T\kern-.1667em\lower.7ex\hbox{E}\kern-.125emX}}
    
\usepackage{hyperref}

\theoremstyle{plain}
\newtheorem{theorem}{Theorem}%[section]
\newtheorem{lemma}[theorem]{Lemma}

\newtheorem{corollary}[theorem]{Corollary}

\theoremstyle{definition}

\newtheorem{remark}[theorem]{Remark}

\newtheorem{example}[theorem]{Example}

\definecolor{darkgreen}{rgb}{0.0, 0.5, 0.0}

\definecolor{brightgreen}{rgb}{0.4, 1.0, 0.0}
\newcolumntype{?}{!{\vrule width 1pt}}

\usepackage[most]{tcolorbox}

\begin{document}

\title{Constructions of Kleene lattices\\
}
\author{
\IEEEauthorblockN{Ivan~Chajda}
\IEEEauthorblockA{\textit{Department of Algebra and Geometry} \\
\textit{Faculty of Science, Palack\'y University Olomouc}\\
Olomouc, Czech Republic\\
ivan.chajda@upol.cz} \\
%\\
\IEEEauthorblockN{Jan~Paseka}
\IEEEauthorblockA{\textit{Department of Mathematics and Statistics} \\
\textit{Faculty of Science, Masaryk University}\\
Brno, Czech Republic\\
paseka@math.muni.cz}

\and
\IEEEauthorblockN{Helmut~L\"anger}
\IEEEauthorblockA{\textit{Institute of Discrete Mathematics and Geometry} \\
\textit{Faculty of Mathematics and Geoinformation, TU Wien}\\
Vienna, Austria, and \\ 
\textit{Department of Algebra and Geometry} \\
\textit{Faculty of Science, Palack\'y University Olomouc}\\
Olomouc, Czech Republic\\
helmut.laenger@tuwien.ac.at}}

\maketitle

\begin{abstract}
We present an easy construction producing a Kleene lattice $\mathbf K=(K,\sqcup,\sqcap,{}')$ from an arbitrary distributive lattice $\mathbf L$ and a non-empty subset of $L$. We show that $\mathbf L$ can be embedded into $\mathbf K$ and compute $|K|$ under certain additional assumptions. We prove that every finite chain considered as a Kleene lattice can be represented in this way and that this construction preserves direct products. Moreover, we demonstrate that certain Kleene lattices that are ordinal sums of distributive lattices are representable. Finally, we prove that not every Kleene lattice is representable.
\end{abstract}


\section{Introduction}

De Morgan lattices, i.e.\ lattices equipped with an antitone involution are usually considered as an algebraic semantics of logics satisfying the double negation law, see e.g.\ \cite{Ci}. Among these lattices so-called {\em Kleene lattices} play a special role. The latter are distributive lattices with an antitone involution $'$ satisfying the so-called {\em normality condition}
\[
x\wedge x'\leq y\vee y'.
\]
Namely, in classical propositional logic as well as in the logic of quantum mechanics the involution ´ is a complementation and therefore 
$x \wedge x' = 0$ and $y \vee y' = 1$. Thus the normality condition is satisfied trivially. Hence, the requirement of normality condition in De Morgan logics 
is a very natural compromise of the lack of this property and it makes this logic closer to the mentioned ones. Also, every MV-algebra is a Kleene lattice. 
Due to this, the question how to construct Kleene lattices is of some interest and importance.

In what follows, we take for granted the concepts and
results on lattices and distributive lattices. For more
information on these topics we direct the reader to 
the monograph \cite{Bi} by G.~Birkhoff.

Kleene lattices were introduced by J.~A.~Kalman (\cite K), see also \cite{Ch} and \cite{CL} for recent results. Recall that an {\em antitone involution} on a poset $(P,\leq)$ is a mapping $':P\rightarrow P$ satisfying
\begin{itemize}
\item $x\leq y$ implies $y'\leq x'$,
\item $x''=x$
\end{itemize}
($x,y\in P$).

It is known that there exists an easy construction producing Kleene lattices from an arbitrary distributive lattice. The construction is as follows. Consider a distributive lattice $\mathbf L=(L,\vee,\wedge)$ and an arbitrary element $a$ of $L$. We can construct the so-called full twist-product of $\mathbf L$. By the {\em full twist-product} of $\mathbf L$ (see e.g.\ \cite{BC} and \cite{TW}) is meant the lattice $(L^2,\sqcup,\sqcap)$ where $\sqcup$ and $\sqcap$ are defined as follows:
\begin{align*}
(x,y)\sqcup(z,v) & :=(x\vee z,y\wedge v), \\
(x,y)\sqcap(z,v) & :=(x\wedge z,y\vee v)
\end{align*}
for all $(x,y),(z,v)\in L^2$. Hence $(x,y)\leq(z,v)$ if and only if both $x\leq z$ and $v\leq y$. Now let $a\in L$ and consider the set
\[
P_a(\mathbf L):=\{(x,y)\in L^2\mid x\wedge y\leq a\leq x\vee y\}.
\]
It was shown in \cite{CL} that
\begin{itemize}
\item $P_a(\mathbf L)$ is a sublattice of the full twist-product of $\mathbf L$,
\item $P_a(\mathbf L)$ is a Kleene lattice where the antitone involution $'$ is defined by $(x,y)':=(y,x)$ for all $(x,y)\in P_a(\mathbf L)$.
\end{itemize}
If $L$ is finite then $P_a(\mathbf L)$ has an odd number of elements. Hence Kleene lattices of even cardinality, e.g.\ the Kleene lattice $\mathbf K$ depicted in Figure~1, cannot be constructed in this way. 

\vspace*{3mm}

\begin{center}
	\setlength{\unitlength}{7mm}
	\begin{tabular}{c c c}
		\begin{picture}(2,4)
			\put(1,0){\circle*{.3}}
			\put(1,1){\circle*{.3}}
			\put(0,2){\circle*{.3}}
			\put(2,2){\circle*{.3}}
			\put(1,3){\circle*{.3}}
			\put(1,4){\circle*{.3}}
			\put(1,1){\line(-1,1)1}
			\put(1,1){\line(0,-1)1}
			\put(1,1){\line(1,1)1}
			\put(1,3){\line(-1,-1)1}
			\put(1,3){\line(1,-1)1}
			\put(1,3){\line(0,1)1}
			\put(.85,-.6){$0$}
			\put(1.35,.8){$a$}
			\put(-.6,1.8){$c$}
			\put(2.35,1.8){$d$}
			\put(1.35,2.8){$b$}
			\put(.85,4.3){$1$}
			\put(.2,-1.4){{\rm Fig.~1}}
		\end{picture}&\phantom{xxccccxxx}&
		\begin{picture}(2,3)
			\put(1,0){\circle*{.3}}
			\put(0,1){\circle*{.3}}
			\put(2,1){\circle*{.3}}
			\put(1,2){\circle*{.3}}
			\put(1,3){\circle*{.3}}
			\put(1,0){\line(-1,1)1}
			\put(1,0){\line(1,1)1}
			\put(1,2){\line(-1,-1)1}
			\put(1,2){\line(1,-1)1}
			\put(1,2){\line(0,1)1}
			\put(.85,-.6){$a$}
			\put(-.6,.8){$c$}
			\put(2.35,.8){$d$}
			\put(1.35,1.8){$b$}
			\put(.85,3.3){$1$}
			\put(.2,-1.4){{\rm Fig.~2}}
		\end{picture}
		
	\end{tabular}
\end{center}

\vspace*{9mm}

This motivated us to search for a more general construction. Namely, 
consider a distributive lattice $\mathbf L=(L,\vee,\wedge)$ and its subset $S$. We put 
\[
P_S(\mathbf L):=\{(x,y)\in L^2\mid x\wedge y\leq z\leq x\vee y\text{ for all }z\in S\}.
\]

It is an easy exercise to show 
that if  $\mathbf L$ is the lattice visualized in Figure~2 and  $S=\{a,b\}$ 
then $\big(P_S(\mathbf L),\sqcup,\sqcap\big)\cong\mathbf K$.

Hence a Kleene lattice $\mathbf K$ will be called {\em representable} if there exists a distributive lattice $\mathbf L=(L,\vee,\wedge)$ and a non-empty subset $S$ of $L$ with $\big(P_S(\mathbf L),\sqcup,\sqcap\big)\cong\mathbf K$.  

Recall that if $S$ is the empty set then $\mathbf P_S(\mathbf L)$ is the classical full twist-product which is not Kleene for every non-trivial distributive lattice 
$\mathbf L$.

Let $(P,\leq)$ be a poset $a,b\in P$ and $A,B\subseteq P$. We say $A\leq B$ if $x\leq y$ for all $x\in A$ and $y\in B$. Instead of $\{a\}\leq \{b\}$, $\{a\}\leq B$ and $A\leq \{b\}$ we simply write $a\leq b$, $a\leq B$ and $A\leq b$, respectively.

\begin{lemma}\label{Lemma1}
Let $\mathbf L=(L,\vee,\wedge)$ be a distributive lattice, $\mathbf L^d=(L,\vee_d,\wedge_d)$ its dual and $S$ a non-empty subset of $L$. Then $\big(P_S(\mathbf L^d),\sqcup_d,\sqcap_d,{}'\big)$ is the dual of $\big(P_S(\mathbf L),\sqcup,\sqcap,{}'\big)$, i.e.\ these Kleene lattices are isomorphic {\rm(}$'$ is an isomorphism{\rm)}.
\end{lemma}

\begin{proof}
Let $a,b\in L$ and $\leq_d$ denote the partial order relation in $\mathbf L^d$. Then
\[
\begin{array}{r c l}
P_S(\mathbf L^d)&=&
\{(x,y)\in L^2\mid x\wedge_dy\leq_dS\leq_dx\vee_dy\}\\
&=&\{(x,y)\in L^2\mid x\vee y\geq S\geq x\wedge y\}=P_S(\mathbf L).
\end{array}
\]

Moreover,
\begin{align*}
a\sqcup_db & =(a\vee_db,a\wedge_db)=(a\wedge b,a\vee b)=a\sqcap b, \\
a\sqcap_db & =(a\wedge_db,a\vee_db)=(a\vee b,a\wedge b)=a\sqcup b.
\end{align*}
\end{proof}

The question is how large the subset $S$ of $L$ should be. We can prove 
 that at most two-element subset S is satisfactory, see the following.

\begin{lemma}\label{lem2}
Let $\mathbf L=(L,\vee,\wedge)$ be a  lattice and $S$ 
a non-empty subset of $L$ such that $\bigwedge S$ and $\bigvee S$ exist. Then 
$P_S(\mathbf L)=P_{\{\bigwedge S,\bigvee S\}}(\mathbf L)$.
\end{lemma}

\begin{proof}
We have
\begin{align*}
P_S(\mathbf L) & =\{(x,y)\in L^2\mid x\wedge y\leq S\leq x\vee y\}\\
               & =\{(x,y)\in L^2\mid x\wedge y\leq\bigwedge S\leq\bigvee S\leq x\vee y\}\\
                & =P_{\{\bigwedge S,\bigvee S\}}(\mathbf L).
\end{align*}
\end{proof}

Hence, if $S$ has both  infimum and supremum  in $L$, we can $S$ restrict to one- or two-element subsets of $L$. The question arises if our construction really produces Kleene lattices. The answer is as follows.

\begin{theorem}\label{th2}
Let $\mathbf L=(L,\vee,\wedge)$ be a distributive lattice and $S$ a non-empty subset of $L$, and put $(x,y)':=(y,x)$ for all $(x,y)\in L^2$. Then
\begin{enumerate}[{\rm(i)}]
\item $\big(P_S(\mathbf L),\sqcup,\sqcap\big)$ is a distributive sublattice of $(L^2,\sqcup,\sqcap)$,
\item $\mathbf P_S(\mathbf L):=\big(P_S(\mathbf L),\sqcup,\sqcap,{}'\big)$ is a Kleene lattice. 
%\item the mapping $x\mapsto(x,a)$ is an embedding of $\mathbf L$ into $(P_S(\mathbf L),\sqcup,\sqcap)$ for all $a\in S$.
\end{enumerate}
\end{theorem}

\begin{proof}
Let $(b,c),(d,e)\in P_S(\mathbf L)$ and $f,g\in L$. Then $(c,b),(e,d)\in P_S(\mathbf L)$.
\begin{enumerate}[(i)]
\item Let $a\in S$. Then in $(L^2,\sqcup,\sqcap)$ we have  $(b,c)\sqcup (d,e)=(b\vee d, c\wedge e)$. We compute
\begin{align*}
(b\vee d)\wedge(c\wedge e) & =\big(b\wedge(c\wedge e)\big)\vee\big(d\wedge(c\wedge e)\big)\\ 
&\leq(b\wedge c)\vee(d\wedge e)\leq a, \\
  (b\vee d)\vee(c\wedge e) & =\big((b\vee d)\vee c\big)\wedge\big((b\vee d)\vee e\big)\\
  &\geq(b\vee c)\vee(d\vee e)\geq a.
\end{align*}
Similarly, $(b,c)\sqcap(d,e)=(b\wedge d,c\vee e)$.

%%%We compute
%%%\begin{align*}
%%%(b\wedge d)\wedge(c\vee e) & =\big((b\wedge d)\wedge c\big)\vee\big((b\wedge d)\wedge e\big)\leq(b\wedge c)\vee(d\wedge e)\leq a, \\
%%%  (b\vee d)\vee(c\wedge e) & =\big((b\vee d)\vee c\big)\wedge\big((b\vee d)\vee e\big)\geq(b\vee c)\wedge(d\vee e)\geq a.
%%%\end{align*}
Hence $(b,c)\sqcap(d,e),(b,c)\sqcup(d,e)\in P_S(\mathbf L)$.
\item The following are equivalent:
$$
\begin{array}{l@{\quad} l @{\quad} l}
 (b,c)\leq(d,e), & b\leq d\text{ and }e\leq c, 
& e\leq c\text{ and }b\leq d, \\[0.2cm]
 (e,d)\leq(c,b), & (d,e)'\leq(b,c)'.
\end{array}
$$
Further, we have $(b,c)''=(c,b)'=(b,c)$. Thus $'$ is an antitone involution on $(P_S(\mathbf L),\sqcup,\sqcap)$. Moreover, there is an element $a\in S$ and we conclude
\begin{align*}
(b,c)\sqcap(b,c)' & =(b,c)\sqcap(c,b)=(b\wedge c,c\vee b)\\ 
 &\leq(a,a)\leq(d\vee e,e\wedge d) \\
                  & =(d,e)\sqcup(e,d)=(d,e)\sqcup(d,e)'
\end{align*}
proving that $\mathbf P_S(\mathbf L)$ is a Kleene lattice.
\end{enumerate}
\end{proof}

The question how to determine such a lattice $\mathbf L$ and 
its subset $S$ to obtain $\mathbf P_S(\mathbf L)$  isomorphic to a given Kleene lattice $\mathbf K$ will be treated in the next sections from several points of view.

%We say that a Kleene lattice $\mathbf K$ is {\em representable} if there exists a distributive lattice $\mathbf L=(L,\vee,\wedge)$ and a non-empty $S$ of $L$ such that $\mathbf P_S(L)\cong\mathbf K$.

\section{Direct products and embeddings} 

In this section we firstly show that the direct product of representable Kleene lattices is representable again. Then we will investigate embeddability of the given distributive lattice $\mathbf L$ into  
$\big(P_S(\mathbf L),\sqcup,\sqcap\big)$  for various subsets $S$.

\begin{theorem}\label{thr}
Let $\mathbf L_i=(L_i,\vee,\wedge)$ be a distributive lattice and $S_i$ a non-empty subset of $L_i$ for every $i\in I$. Put
\[
\mathbf L:=\prod_{i\in I}\mathbf L_i\text{ and }S:=\prod_{i\in I}S_i.
\]
Then
\[
\mathbf P_S(\mathbf L)\cong\prod_{i\in I}\mathbf P_{S_i}(\mathbf L_i).
\]
\end{theorem}
 
\begin{proof}
Let us show that the mapping $f$ from $\prod\limits_{i\in I}P_{S_i}(\mathbf L_i)$ to $P_S(\mathbf L)$ defined by
\[
f\big((x_i,y_i)_{i\in I}\big):=\big((x_i)_{i\in I},(y_i)_{i\in I}\big)
\]
for all $(x_i,y_i)_{i\in I}\in\prod\limits_{i\in I}P_{S_i}(\mathbf L_i)$ is an isomorphism from $\prod\limits_{i\in I}\mathbf P_{S_i}(\mathbf L_i)$ to $\mathbf P_S(\mathbf L)$. \\
First, we have to check that $\big((x_i)_{i\in I},(y_i)_{i\in I}\big)\in P_S(\mathbf L)$. \\
Evidently, $S\not=\emptyset$. Suppose $a=(a_i)_{i\in I}\in S$. Then $a_i\in S_i$ for every $i\in I$. Since $(x_i,y_i)\in P_{S_i}(\mathbf L_i)$ we have that $x_i\wedge y_i\leq a_i\leq x_i\vee y_i$. Hence $(x_i)_{i\in I}\wedge(y_i)_{i\in I}\leq a\leq(x_i)_{i\in I}\vee(y_i)_{i\in I}$. \\
Second, let us check that $f$ is an order embedding. Assume $(x_i,y_i)_{i\in I},(u_i,v_i)_{i\in I}\in\prod\limits_{i\in I}P_{S_i}(\mathbf L_i)$. Suppose $(x_i,y_i)_{i\in I}\leq(u_i,v_i)_{i\in I}$. Then $x_i\leq u_i$ and $v_i\leq y_i$ for every $i\in I$. Hence $(x_i)_{i\in I}\leq(u_i)_{i\in I}$ and $(v_i)_{i\in I}\leq(y_i)_{i\in I}$ in $\mathbf L$. We conclude that $\big((x_i)_{i\in I},(y_i)_{i\in I}\big)\leq\big((u_i)_{i\in I},(v_i)_{i\in I}\big)$ in $\mathbf P_S(\mathbf L)$. Conversely, if $\big((x_i)_{i\in I},(y_i)_{i\in I}\big)\leq\big((u_i)_{i\in I},(v_i)_{i\in I}\big)$ in $\mathbf P_S(\mathbf L)$ then $(x_i)_{i\in I}\leq(u_i)_{i\in I}$ and $(v_i)_{i\in I}\leq(y_i)_{i\in I}$ in $\mathbf L$. Therefore $(x_i,y_i)_{i\in I}\leq(u_i,v_i)_{i\in I}$ in $\prod\limits_{i\in I}\mathbf P_{S_i}(\mathbf L_i)$. \\
Third, we have to verify that $f$ is surjective. Let $z\in P_S(\mathbf L)$. Then $z=(x,y)\in L^2$ and $x\wedge y\leq a\leq x\vee y$ for all $a\in S$. We have $x=(x_i)_{i\in I}$ and $y=(y_i)_{i\in I}$ with $x_i,y_i\in L_i$ for every $i\in I$. Let us check that $(x_i,y_i)\in P_{S_i}(\mathbf L_i)$ for every $i\in I$. Let $a_i\in S_i$  for every $i\in I$ (this is possible since all $S_i$ are non-empty). Then $a:=(a_i)_{i\in I}\in S$ and we have
\[
(x_i)_{i\in I}\wedge(y_i)_{i\in I}\leq(a_i)_{i\in I}\leq(x_i)_{i\in I}\vee(y_i)_{i\in I}.
\]
Therefore $x_i\wedge y_i\leq a_i\leq x_i\vee y_i$ for all $a_i\in S_i$ and $i\in I$. \\
Fourth, let us show that $f$ preserves $'$. Assume that $(x_i,y_i)_{i\in I}\in\prod\limits_{i\in I} P_{S_i}(\mathbf L_i)$. Then $(y_i,x_i)_{i\in I}\in\prod\limits_{i\in I} P_{S_i}(\mathbf L_i)$ and $\big((x_i)_{i\in I},(y_i)_{i\in I}\big),\big((y_i)_{i\in I},(x_i)_{i\in I}\big)\in P_S(\mathbf L)$. We compute
\begin{align*}
f\Big(\big((x_i,y_i)_{i\in I}\big)'\Big) & =f\big((y_i,x_i)_{i\in I}\big)=\big((y_i)_{i\in I},(x_i)_{i\in I}\big)\\
& =\big((x_i)_{i\in I},(y_i)_{i\in I}\big)' =\Big(f\big((x_i,y_i)_{i\in I}\big)\Big)'.
\end{align*}
\end{proof}

\begin{corollary}\label{correp}
Direct products of representable Kleene lattices are representable Kleene lattices.
\end{corollary}

\begin{lemma}\label{embedd} 
	Let $\mathbf L=(L,\vee,\wedge)$ be a distributive lattice and $a,b\in L$ with $a\leq b$ and assume that there exists an antitone complementation on $([a,b],\vee,\wedge)$ and 
	$a<(x\wedge b)\vee a < b$ implies $(x\wedge b)\vee a=x$. Then $\mathbf L$ can be embedded into $(P_{\{a, b\}}(\mathbf L),\sqcup,\sqcap)$.
\end{lemma}

\begin{proof} Let $'$ be an antitone complementation on $([a,b],\vee,\wedge)$ and define $f:L\rightarrow P_{\{a, b\}}(\mathbf L)$ as follows:
\[
f(x):=\left\{
\begin{array}{ll}
(x,b)  & \text{if }x\wedge b\leq a, \\
(x,a)  & \text{if }b\leq x \vee a,\\
(x, x') & \text{otherwise,} \\
\end{array}
\right.
\]
$x\in L$. Of course, $a'=b$ and $b'=a$, and $x\in[a,b]$ implies $x'\in[a,b]$ 
and $f(x)=(x, x')$. 
If $x\wedge b\leq a$ and $b\leq x \vee a$ then $a=b$. 
Namely, $b\leq x \vee a$ implies $b= (x \vee a)\wedge b=(x\wedge b)\vee a\leq a$ . 
If $x\wedge b\leq a$ or $b\leq x \vee a$ our definition is clearly correct. 
Assume that $x\in L$ such that $x\wedge b\not\leq a$ and 
$b\not\leq x \vee a$. Then $a < (x\wedge b)\vee a= (x \vee a)\wedge b=x < b$ and again  
$f$ is defined correctly.

It is evident that $f$ preserves binary meets and joins. 
Hence  $f$  is also order-preserving.
Namely, let $I=\{x\in L \mid x\wedge b\leq a\}$ and 
$F=\{x\in L \mid b\leq x\vee a\}$. Then $I\cup F\cup [a,b]=L$, 
$I$ is an ideal in  $\mathbf L$ and $F$ is a filter in  $\mathbf L$. 
Moreover, $a\in I$ and $b\in F$.

If $x, y\in I$ or $x, y\in F$ or  $x, y\in [a,b]$ then, 
since $I$ , $F$ and $[a,b]$ are closed under finite meets and joins, we have 
$f(x)\sqcup f(y)=f(x\vee y)$ and $f(x)\sqcap f(y)=f(x\wedge y)$.

Assume now that $x\in I$ and $y\in F$. Then 
$f(x)=(x,b)$, $f(y)=(y,a)$, 
$x\wedge y\wedge b\leq a$ and 
$x\vee y\vee a\geq b$. Hence also $f(x\wedge y)=(x\wedge y, b)$ 
and $f(x\vee y)=(x\vee y, a)$. We compute:
$$
\begin{array}{c}
f(x\wedge y)=(x\wedge y, b)=(x, b)\sqcap (y,a)=f(x)\sqcap f(y),\\
f(x\vee y)=(x\vee y, a)=(x, b)\sqcup (y,a)=f(x)\sqcup f(y). 
\end{array}
$$
From now, we may assume that $a< b$ (for $a=b$ the statement is evident 
since $I=F=L$). 
Suppose now $x\in I$ and $y\in [a,b]\setminus I$. Then $x\wedge y\in I$, 
$x\vee y\not\in I$, $y\not\in F\setminus \{b\}$ and $a< x\vee y$. 
Namely, since $a\in I$ we have $a<y\leq x\vee y$. Similarly, 
since $y\wedge b\not\leq a$ we obtain that 
$(x\vee y)\wedge b\not\leq a$, i.e., $x\vee y\not\in I$. 
Assume now that $y\in F$. Then $b\leq y\vee a=y\leq b$, i.e., 
$y\not\in F\setminus \{b\}$. 
 We now compute:
$$
\begin{array}{c}
	f(x\wedge y)=(x\wedge y, b)=(x, b)\sqcap (y,y')=f(x)\sqcap f(y).
\end{array}
$$
Assume first $x\vee y\not\in F$. Then $a< y$, $y\not\in (I\cup F)\setminus \{b\}$. 
Hence $a<(x\vee y\vee a)\wedge b=\left((x\vee y)\wedge b\right)\vee a<b$ and by assumption of the Lemma 
we have
\vskip-0.8cm
$$
\begin{array}{c}
x\vee y= \big((x\vee y)\wedge b\big)\vee a=%
(\overbracket{x\wedge b}^{\leq a})\vee (\overbracket{y\wedge b}^{=y})\vee a =y, \text{ and} 
\end{array}
$$
since $x\in I$ and $y\not\in I\cup F$ we have that $(x,b)=f(x)$ and  $(y,y')=f(y)=f(x\vee y)$. We conclude 
$$
\begin{array}{c}
f(x\vee y)=(y, y')=(x, b)\sqcup (y,y')=f(x)\sqcup f(y). 
\end{array}
$$
Now, let $x\vee y\in F$. Then in fact $x\vee y=(x\vee y)\vee a\geq b$. Since $a\leq y$ and $a\leq b$ 
we have 
$b=(x\wedge b)\vee (y\wedge b)\leq a \vee (y\wedge b)=y\wedge b\leq y\leq b$. Therefore  $y=b$ and 
$$
\begin{array}{c}
	f(x\vee y)=(x\vee b, a)=(x, b)\sqcup (b,a)=f(x)\sqcup f(y). 
\end{array}
$$
The case $x\in F$ and $y\in [a,b]\setminus F$ can be verified similarly.

Now, let $x, y\in L$ such that 
$f(x)\sqsubseteq f(y)$. Since the first coordinate of $f(x)$ is $x$ we have that $x\leq y$. Hence $f$ reflects order. 
\end{proof}

Note that if $'$ is an antitone complementation on $([a,b],\vee,\wedge)$ then 
$([a,b],\vee,\wedge,{}')$ is a Boolean algebra, and conversely. 

In the following we write $P_a(\mathbf L)$ and $P_{ab}(\mathbf L)$ instead of $P_{\{a\}}(\mathbf L)$ and $P_{\{a,b\}}(\mathbf L)$, respectively.

\begin{corollary}
Let $\mathbf L=(L,\vee,\wedge)$ be a distributive lattice and $a\in L$. Then $\mathbf L$ can be embedded into $(P_a(\mathbf L),\sqcup,\sqcap)$ via the  prescription $x\mapsto (x,a)$.
\end{corollary} 

\begin{corollary}
	Let $\mathbf L=(L,\vee,\wedge)$ be a distributive lattice and $a,b\in L$ with $a<b$ such that $b$ covers $a$. Then $\mathbf L$ can be embedded into $(P_{ab}(\mathbf L),\sqcup,\sqcap)$ via the prescription 
	$$
	x\mapsto \left\{
	\begin{array}{ll}
		(x,b)  & \text{if }x\leq a, \\
		(x,a)  & \text{if }b\leq x.
	\end{array}
	\right.
	$$
\end{corollary}

\begin{corollary}
Let $\mathbf L=(L,\vee,\wedge)$ be a distributive lattice and $a,b\in L$ with $a<b$ and $L=(a]\cup[a,b]\cup[b)$ and assume that there exists an antitone complementation on $([a,b],\vee,\wedge)$. Then $\mathbf L$ can be embedded into $(P_{ab}(\mathbf L),\sqcup,\sqcap)$ via the prescription 
$$
x\mapsto \left\{
\begin{array}{ll}
(x,b)  & \text{if }x\leq a, \\
(x,x') & \text{if }a\leq x\leq b, \\
(x,a)  & \text{if }b\leq x.
\end{array}
\right.
$$
\end{corollary}

\begin{remark}
If $\mathbf L$ can be embedded into $\big(P_S(\mathbf L),\sqcup,\sqcap\big)$ and 
has a chain or antichain of a certain cardinality then $\big(P_S(\mathbf L),\sqsubseteq\big)$ 
has a chain or antichain of the same cardinality, respectively. For instance, 
if $\big(P_S(\mathbf L),\sqsubseteq\big)$ has no three-element antichain then the same holds for $\mathbf L$.
\end{remark}

\begin{example}\label{ex1}
If $\mathbf L$ is the lattice shown in Figure~1 then $\big(P_a(\mathbf L),\sqcup,\sqcap\big)$ is depicted in Figure~3. The non-filled circles 
indicate the embedding of $\mathbf L$ into 
$\big(P_a(\mathbf L),\sqcup,\sqcap\big)$.
\end{example}

%\vspace*{-2mm}

\begin{center}
\setlength{\unitlength}{5.0mm}
\begin{picture}(10,18)
\put(3,1){\circle*{.3}}
\put(3,3){\circle*{.3}}
\put(5,3){\circle*{.3}}
\put(3,5){\circle*{.3}}
\put(5,5){\circle*{.3}}
\put(7,5){\circle*{.3}}
\put(3,7){\circle*{.3}}
\put(5,7){\circle{.3}}
\put(7,7){\circle*{.3}}
\put(1,9){\circle*{.3}}
\put(5,9){\circle{.3}}
\put(9,9){\circle*{.3}}
\put(3,11){\circle{.3}}
\put(5,11){\circle*{.3}}
\put(7,11){\circle{.3}}
\put(3,13){\circle*{.3}}
\put(5,13){\circle{.3}}
\put(7,13){\circle*{.3}}
\put(5,15){\circle*{.3}}
\put(7,15){\circle{.3}}
\put(7,17){\circle*{.3}}
\put(3,1){\line(0,1)2}
\put(3,1){\line(1,1)4}
\put(3,3){\line(1,1)6}
\put(5,3){\line(-1,1)2}
\put(5,3){\line(0,1)2}
\put(3,5){\line(0,1)2}
\put(3,5){\line(1,1)2}
\put(5,5){\line(-1,1)4}
\put(7,5){\line(-1,1)2}
\put(7,5){\line(0,1)2}
\put(3,7){\line(1,1)4}
\put(5,7){\line(0,1)4}
\put(7,7){\line(-1,1)4}
\put(1,9){\line(1,1)6}
\put(9,9){\line(-1,1)4}
\put(3,11){\line(0,1)2}
\put(5,11){\line(-1,1)2}
\put(5,11){\line(1,1)2}
\put(7,11){\line(0,1)2}
\put(3,13){\line(1,1)4}
\put(5,13){\line(0,1)2}
\put(7,13){\line(-1,1)2}
\put(7,15){\line(0,1)2}
\put(2.35,.25){$(0,1)$}
\put(1.2,2.85){$(a,1)$}
\put(5.4,2.85){$(0,b)$}
\put(1.2,4.85){$(0,c)$}
\put(5.25,4.85){$(a,b)$}
\put(7.4,4.85){$(0,d)$}
\put(1.2,6.85){$(a,c)$}
\put(5.25,6.85){$(0,a)$}
\put(7.4,6.85){$(a,d)$}
\put(-.8,8.85){$(d,c)$}
\put(5.4,8.85){$(a,a)$}
\put(9.4,8.85){$(c,d)$}
\put(1.2,10.85){$(d,a)$}
\put(3.225,10.85){$(a,0)$}
\put(7.4,10.85){$(c,a)$}
\put(1.2,12.85){$(d,0)$}
\put(3.25,12.85){$(b,a)$}
\put(7.4,12.85){$(c,0)$}
\put(3.2,14.85){$(b,0)$}
\put(7.4,14.85){$(1,a)$}
\put(6.35,17.4){$(1,0)$}
\put(4.2,-.75){{\rm Fig.~3}}
\end{picture}
\end{center}

\vspace*{5mm}

\begin{example}
	If $\mathbf L$ is the lattice shown in Figure~1 then $\big(P_{ab}(\mathbf L),\sqcup,\sqcap\big)$ is visualized in Figure~4.	
	The non-filled circles indicate the embedding of $\mathbf L$ into $\big(P_{ab}(\mathbf L),\sqcup,\sqcap\big)$.
\end{example}

\vspace*{-2mm}

\begin{center}
	\setlength{\unitlength}{5.0mm}
	\begin{picture}(6,14)
		\put(3,1){\circle*{.3}}
		\put(1,3){\circle*{.3}}
		\put(5,3){\circle{.3}}
		\put(3,5){\circle{.3}}
		\put(1,7){\circle{.3}}
		\put(5,7){\circle{.3}}
		\put(3,9){\circle{.3}}
		\put(1,11){\circle*{.3}}
		\put(5,11){\circle{.3}}
		\put(3,13){\circle*{.3}}
		\put(3,1){\line(-1,1)2}
		\put(3,1){\line(1,1)2}
		\put(1,3){\line(1,1)4}
		\put(5,3){\line(-1,1)4}
		\put(1,7){\line(1,1)4}
		\put(5,7){\line(-1,1)4}
		\put(1,11){\line(1,1)2}
		\put(5,11){\line(-1,1)2}
		\put(2.35,.25){$(0,1)$}
		\put(-.8,2.85){$(a,1)$}
		\put(5.4,2.85){$(0,b)$}
		\put(3.4,4.85){$(a,b)$}
		\put(-.8,6.85){$(c,d)$}
		\put(-.8,10.85){$(b,0)$}
		\put(3.4,8.85){$(b,a)$}
		\put(5.4,6.85){$(d,c)$}
		\put(5.4,10.85){$(1,a)$}
		\put(2.35,13.4){$(1,0)$}
		\put(2.2,-.75){{\rm Fig.~4}}
	\end{picture}
\end{center}

\vspace*{4mm}

\section{Cardinality of $P_S(\mathbf L)$}

If a distributive lattice $\mathbf L$  is finite then 
also $P_S(\mathbf L)$ is finite and hence it makes sense to ask about its cardinality depending on the cardinality of $\mathbf L$. Using 
results reached here we will be able to state and to prove that not every Kleene lattice is representable.

\begin{lemma}
	Let $\mathbf L=(L,\vee,\wedge)$ be a finite lattice and $a\in L$ with $L=[0,a]\cup[a,1]$. Then
	$$
	\begin{array}{@{}l@{}}
		|P_a(\mathbf L)| =|\{(x,y)\in L^2\mid x,y<a;x\vee y=a\}|-1\\
		\multicolumn{1}{r}{\phantom{x}+|\{(x,y)\in L^2\mid x,y>a;x\wedge y=a\}|+ 2|[0,a]|\cdot|[a,1]|.}
	\end{array}
	$$
\end{lemma}

\begin{proof}
Put $b:=|[0,a]|$ and $c:=|[a,1]|$. Then
$$
\begin{array}{l r}
|\{(x,y)\in P_a(\mathbf L)\mid x, y<a\}| &\phantom{xxxxxxxxxxxxxxx} \\
\multicolumn{2}{r}{=|\{(x,y)\in L^2\mid x,y<a;x\vee y=a\}|,} \phantom{\text{ and}}\\
|\{(x,y)\in P_a(\mathbf L)\mid x>a;y>a\}| & \\
\multicolumn{2}{r}{=|\{(x,y)\in L^2\mid x,y>a;x\wedge y=a\}|, \text{ and}}
\end{array}
$$
\begin{align*}
|\{(x,y)\in P_a(\mathbf L)\mid x<a;y=a\}| & =b-1, \\
|\{(x,y)\in P_a(\mathbf L)\mid x<a;y>a\}| & =(b-1)(c-1), \\
|\{(x,y)\in P_a(\mathbf L)\mid x=a;y<a\}| & =b-1, \\
|\{(x,y)\in P_a(\mathbf L)\mid x=a;y=a\}| & =1, \\
|\{(x,y)\in P_a(\mathbf L)\mid x=a;y>a\}| & =c-1, \\
|\{(x,y)\in P_a(\mathbf L)\mid x>a;y<a\}| & =(b-1)(c-1), \\
|\{(x,y)\in P_a(\mathbf L)\mid x>a;y=a\}| & =c-1.\\
\end{align*}

\vspace*{-4mm}

Adding up these numbers yields
$$
\begin{array}{@{}l@{}}
|P_a(\mathbf L)| =|\{(x,y)\in L^2\mid x,y<a;x\vee y=a\}|-1\\
\multicolumn{1}{r}{\phantom{x}+|\{(x,y)\in L^2\mid x,y>a;x\wedge y=a\}|+ 2bc.}
\end{array}
$$
\end{proof}

\begin{lemma}\label{lem4}
Let $\mathbf L=(L,\vee,\wedge)$ be a finite lattice and $a\in L$. Then
\begin{enumerate}[{\rm(i)}]
\item $|P_a(\mathbf L)|$ is odd,
\item $|P_a(\mathbf L)|\geq2|L|-1$,
\item $|L|\leq(|P_a(\mathbf L)|+1)/2$,
\item $|P_a(\mathbf L)|\geq2|[0,a]|\cdot|[a,1]|-1$.
\end{enumerate}
\end{lemma}

\begin{proof}
\
\begin{enumerate}[(i)]
\item For every $x,y\in L$ we have
\begin{align*}
(x,y)\in P_a(\mathbf L) & \Leftrightarrow(y,x)\in P_a(\mathbf L), \\
(x,x)\in P_a(\mathbf L) & \Leftrightarrow x=a.
\end{align*}
\item This follows from $P_a(\mathbf L)\supseteq(\{a\}\times L)\cup(L\times\{a\})$.
\item This follows from (ii).
\item This follows from $P_a(\mathbf L)\supseteq([0,a]\times[a,1])\cup([a,1]\times[0,a])$.
\end{enumerate}
\end{proof}

\begin{lemma}
Let $\mathbf L=(L,\vee,\wedge)$ be an infinite lattice and $a\in L$. Then $|P_a(\mathbf L)|=|L|$.
\end{lemma}

\begin{proof}
We have $L\times\{a\}\subseteq P_a(\mathbf L)\subseteq L^2$ and hence
\[
|L|=|L\times\{a\}|\leq|P_a(\mathbf L)|\leq|L^2|=|L|^2=|L|.
\]
\end{proof}

\begin{lemma}
Let $\mathbf L=(L,\vee,\wedge)$ be a finite lattice and 
$a,b\in L$ with $a<b$ and $L=[0,a]\cup[a,b]\cup[b,1]$. Then
\[
\begin{array}{@{}r c l@{}}
|P_{ab}(\mathbf L)|&=&2|[0,a]|\cdot|[b,1]|\\
\multicolumn{3}{l}{%
\phantom{xxxi}+|\{(x,y)\in L^2\mid a<x,y<b;x\wedge y=a;x\vee y=b\}|.}
\end{array}
\]
\end{lemma}

\begin{proof}
Put $c:=|[0,a]|$ and $d:=|[b,1]|$. Then
$$
\begin{array}{@{}l r}
|\{(x,y)\in P_{ab}(\mathbf L)\mid a<x<b;a<y<b\}|& \\
\multicolumn{2}{r}{=|\{(x,y)\in L^2\mid a<x,y<b;x\wedge y=a; x\vee y=b\}|, \text{ and}}
\end{array}
$$
\begin{align*}
    |\{(x,y)\in P_{ab}(\mathbf L)\mid x\leq a;y<b\}| & =0, \\
|\{(x,y)\in P_{ab}(\mathbf L)\mid x\leq a;y\geq b\}| & =cd, \\
  |\{(x,y)\in P_{ab}(\mathbf L)\mid a<x<b;y\leq a\}| & =0, \\
  |\{(x,y)\in P_{ab}(\mathbf L)\mid a<x<b;y\geq b\}| & =0, \\
|\{(x,y)\in P_{ab}(\mathbf L)\mid x\geq b;y\leq a\}| & =cd, \\
    |\{(x,y)\in P_{ab}(\mathbf L)\mid x\geq b;y>a\}| & =0.
\end{align*}
Adding up these numbers yields
\[
\begin{array}{@{}r c l@{}}
|P_{ab}(\mathbf L)|&=&2cd\\
\multicolumn{3}{l}{%
\phantom{xxxi}+|\{(x,y)\in L^2\mid a<x,y<b;x\wedge y=a;x\vee y=b\}|.}
\end{array}
\]
\end{proof}

\begin{lemma}\label{lem3}
Let $\mathbf L=(L,\vee,\wedge)$ be a finite lattice and $a,b\in L$ with $a<b$. Then
\begin{enumerate}[{\rm(i)}]
\item $|P_{ab}(\mathbf L)|$ is even,
\item $|P_{ab}(\mathbf L)|\geq2|[0,a]|\cdot|[b,1]|$.
\end{enumerate}
\end{lemma}

\begin{proof}
\
\begin{enumerate}[(i)]
\item For every $x,y\in L$ we have
\begin{align*}
& (x,y)\in P_{ab}(\mathbf L)\Leftrightarrow(y,x)\in P_{ab}(\mathbf L), \\
& (x,x)\notin P_{ab}(\mathbf L).
\end{align*}
\item This follows from $P_{ab}(\mathbf L)\supseteq([0,a]\times[b,1])\cup([b,1]\times[0,a])$.
\end{enumerate}
\end{proof}

\begin{lemma}\label{lem1}
Let $\mathbf L=(L,\vee,\wedge)$ be a chain and $a\in L$. Then $P_a(\mathbf L)=\big((a]\times[a)\big)\cup\big([a)\times(a]\big)$.
\end{lemma}

\begin{proof}
We have
$$
\begin{array}{r@{\,}c@{\,} l}
P_a(\mathbf L) & =&\{(x,y)\in L^2\mid x\wedge y\leq a\leq x\vee y\}\\
\multicolumn{3}{r}{\phantom{x}=\{(x,y)\in L^2\mid x\leq a\leq y\}\cup\{(x,y)\in L^2\mid y\leq a\leq x\}}\\
							 &=&\big((a]\times[a)\big)\cup\big([a)\times(a]\big).
\end{array}
$$
\end{proof}

\begin{lemma}
Let $\mathbf L=(L,\vee,\wedge)$ be a finite chain and $a\in L$. 
Then $|P_a(\mathbf L)|=2|[0,a]|\cdot|[a,1]|-1$.
\end{lemma}

\begin{proof}
This follows from Lemma~\ref{lem1}.
\end{proof}

\section{Representation of Kleene lattices}

In the following let $(a_1<\cdots<a_n)$ denote an $n$-element chain.

\begin{lemma}
For every non-negative integer $n$ we have
\begin{align*}
     P_a(a<a_1<\cdots<a_n) & =\big((a,a_n)\sqsubset\cdots\sqsubset(a,a_1)\sqsubset(a,a)\\
     & \sqsubset(a_1,a)\sqsubset\cdots\sqsubset(a_n,a)\big), \\
\end{align*}    
\vskip-1.343cm
\begin{align*} 
P_{ab}(a<b<&\, \, a_1<\cdots<a_n)  =\big((a,a_n)\sqsubset\cdots\sqsubset(a,a_1)\\
     &\sqsubset(a,b)\sqsubset(b,a)\sqsubset(a_1,a)\sqsubset\cdots\sqsubset(a_n,a)\big).
\end{align*}
\end{lemma}

\begin{proof}
The result is easy to verify.
\end{proof}

In the following for every positive integer $n$ let $\mathbf C_n=(C_n,\vee,\wedge)$ denote an $n$-element chain. It is easy to see that there exists exactly one antitone involution on $\mathbf C_n$, and $\mathbf C_n$ together with this antitone involution forms a Kleene lattice.

\begin{corollary}
For every positive integer $n$ we have
\[
(C_n,\vee,\wedge,{}')\cong\big(P_S(\mathbf C_{\lfloor n/2\rfloor+1}),\sqcup,\sqcap,{}'\big)
\]
where $S$ denotes the set consisting of the smallest element of $\mathbf C_{\lfloor n/2\rfloor+1}$ if $n$ is odd and $S$ denotes the set consisting of the two smallest elements of $\mathbf C_{\lfloor n/2\rfloor+1}$ if $n$ is even.
\end{corollary}

\begin{corollary}
If $(n_i;i\in I)$ is a non-void family of positive integers then
\[
\prod_{i\in I}(C_{n_i},\vee,\wedge,{}')\cong\big(P_{\prod\limits_{i\in I}S_i}(\prod_{i\in I}\mathbf C_{\lfloor n_i/2\rfloor+1}),\sqcup,\sqcap,{}'\big)
\]
where for every $i\in I$, $S_i$ denotes the set consisting of the smallest element of $\mathbf C_{\lfloor n_i/2\rfloor+1}$ if $n_i$ is odd and $S_i$ denotes the set consisting of the two smallest elements of $\mathbf C_{\lfloor n_i/2\rfloor+1}$ if $n_i$ is even.
\end{corollary}

\begin{example}
The four-element chain $\mathbf C_4$ is as a Kleene lattice isomorphic to
\[
\big(P_{ab}(a<b<c),\sqcup,\sqcap,{}'\big).
\]
Hence the direct product of the Kleene lattices $\mathbf C_4$ and $\mathbf C_4$ is isomorphic to
\[
\Big(P_{\{(a,a),(b,b)\}}\big((a<b<c)\times(a<b<c)\big),\sqcup,\sqcap,{}'\Big).
\]
\end{example}

In the sequel we use the following notation: 
If $\mathbf L_1$ and $\mathbf L_2$ are lattices with a greatest and a smallest element, respectively, then by $\mathbf L_1+_a\mathbf L_2$ we denote the ordinal sum of $\mathbf L_1$ and $\mathbf L_2$ where the greatest element $a$ of $\mathbf L_1$ is identified with the smallest element of $\mathbf L_2$. If $\mathbf L_1$, $\mathbf L_2$ and $\mathbf L_3$ are lattices with a greatest, with a smallest and a greatest and with a smallest element, respectively, then by $\mathbf L_1+_a\mathbf L_2+_b\mathbf L_3$ we denote the ordinal sum of $\mathbf L_1$, $\mathbf L_2$ and $\mathbf L_3$ where the greatest element $a$ of $\mathbf L_1$ is identified with the smallest element of $\mathbf L_2$ and the greatest element $b$ of $\mathbf L_2$ is identified with the smallest element of $\mathbf L_3$.  
Finally, let $\mathbf1$ denote the trivial Boolean algebra.

\begin{lemma}\label{lemmalbl}
	Let $\mathbf L_1=(L_1,\vee,\wedge)$ and $\mathbf L_2=(L_2,\vee,\wedge)$ be 
	distributive lattices  with greatest element $a$ and smallest element $b$, respectively, and $\mathbf B=(B,\vee,\wedge,{}') $ a  Boolean algebra with smallest element $a$ and greatest element $b$. If $a< b$ or 
	$a=b$ is join-irreducible in $\mathbf L_1$ and meet-irreducible 
	in $\mathbf L_2$ then
	$$
	\begin{array}{@{}c@{}}
		\big(P_{ab}(\mathbf L_1+_a\mathbf B+_b\mathbf L_2),\sqcup,\sqcap\big)%
		\cong \phantom{xxxxxxxxxxxxx\mathbf L_1+_a\mathbf B+_b\mathbf L_2)}\\[0.2cm]%
		\phantom{xxxxx\mathbf L_1+_a\mathbf B+_b\mathbf L_2)}%
		(\mathbf L_1\times\mathbf L_2^d)+_{(a,b)}\mathbf B+_{(b,a)}%
		(\mathbf L_2\times\mathbf L_1^d).\phantom{xxxx}
	\end{array}
	$$
\end{lemma}

\begin{proof}
	We have
	\[
	P_{ab}(\mathbf L_1+_a\mathbf B+_b\mathbf L_2)=(L_1\times L_2)\cup\{(x,x')\mid x\in B\}\cup(L_2\times L_1).
	\]
\end{proof}

\begin{corollary}\label{lemmacbc}
	Let $\mathbf L_1=(L_1,\vee,\wedge)$ and $\mathbf L_2=(L_2,\vee,\wedge)$ be 
	chains  with greatest element $a$ and smallest element $b$, respectively, and $\mathbf B=(B,\vee,\wedge,{}') $ a  Boolean algebra with smallest element $a$ and greatest element $b$. Then
	$$
	\begin{array}{@{}c@{}}
		\big(P_{ab}(\mathbf L_1+_a\mathbf B+_b\mathbf L_2),\sqcup,\sqcap\big)%
		\cong \phantom{xxxxxxxxxxxxx\mathbf L_1+_a\mathbf B+_b\mathbf L_2)}\\[0.2cm]%
		\phantom{xxxxx\mathbf L_1+_a\mathbf B+_b\mathbf L_2)}%
		(\mathbf L_1\times\mathbf L_2^d)+_{(a,b)}\mathbf B+_{(b,a)}%
		(\mathbf L_2\times\mathbf L_1^d).\phantom{xxxx}
	\end{array}
	$$
\end{corollary}

\begin{corollary}\label{lem5a}
	Let $\mathbf L_1=(L_1,\vee,\wedge)$ and $\mathbf L_2=(L_2,\vee,\wedge)$ be distributive lattices with greatest element $a$ and smallest element $b$, respectively, and $\mathbf B=(B,\vee,\wedge,{}')$ a  Boolean algebra with smallest element $a$ and greatest element $b$. If $a< b$ or 
	$a=b$ is join-irreducible in $\mathbf L_1$ and meet-irreducible 
	in $\mathbf L_2$ then 
	\begin{align*}
		\big(P_{ab}(\mathbf L_1+_a\mathbf B),\sqcup,\sqcap\big) & \cong\mathbf L_1+_{(a,b)}\mathbf B+_{(b,a)}\mathbf L_1^d, \\
		\big(P_{ab}(\mathbf B+_b\mathbf L_2),\sqcup,\sqcap\big) & \cong\mathbf L_2^d+_{(a,b)}\mathbf B+_{(b,a)}\mathbf L_2.
	\end{align*}
\end{corollary}

\begin{proof} It is enough to put $\mathbf L_2:=\mathbf 1$ or 
	 $\mathbf L_1:=\mathbf 1$ and use 
	Lemma \ref{lemmalbl}. 
\end{proof}	

\begin{corollary}\label{lem5}
Let $\mathbf L_1=(L_1,\vee,\wedge)$ and $\mathbf L_2=(L_2,\vee,\wedge)$ be 
distributive lattices  with greatest and smallest element $a$, respectively. 
If $a$ is join-irreducible in $\mathbf L_1$ and meet-irreducible 
in $\mathbf L_2$ then 
\[
\begin{array}{r@{\,\,}c@{\,\,}l}
\big(P_a(\mathbf L_1+_a\mathbf L_2),\sqcup,\sqcap\big)&\cong&(\mathbf L_1\times\mathbf L_2^d)+_{(a,a)}(\mathbf L_2\times\mathbf L_1^d),\\[0.2cm]
\big(P_a ({\mathbf L}_1),\sqcup,\sqcap\big)%
&\cong&\mathbf L_1+_{(a,a)}\mathbf L_1^d,\\[0.2cm]
\big(P_a(\mathbf L_2),\sqcup,\sqcap\big)%
&\cong&\mathbf L_2^d+_{(a,a)}\mathbf L_2.
\end{array}
\]
\end{corollary}
\begin{proof} It is enough to put $\mathbf B:=\mathbf 1$ and use 
		Lemma \ref{lemmalbl} or Corollary \ref{lem5a}. 
\end{proof}

Finally, we present a non-representable Kleene lattice.

\begin{lemma}\label{Lemma3c} 
Let $\mathbf L=(L,\vee,\wedge)$ be a distributive lattice and $a\in L$ and assume 
$\big(P_a(\mathbf L),\sqsubseteq)$ to have no three-element antichain. Then $a$ is 
comparable with every element of $L$ and join- and meet-irreducible.
\end{lemma}

\begin{proof}
If there would exist some element $b$ of $L$ with $b\parallel a$ then $\{(a,a),(a,b),(b,a)\}$ 
would be a three-element antichain of $\big(P_a(\mathbf L),\sqsubseteq)$. If $a$ would 
not be join-irreducible then there would exist $c,d\in L\setminus\{a\}$ with $c\vee d=a$ 
and $\{(a,a),(c,d),(d,c)\}$ would be a three-element antichain of $\big(P_a(\mathbf L),\sqsubseteq)$. 
If $a$ would not be meet-irreducible then there would exist $e,f\in L\setminus\{a\}$ with 
$e\wedge f=a$ and $\{(a,a),(e,f),(f,e)\}$ would be a three-element antichain of $\big(P_a(\mathbf L),\sqsubseteq)$.
\end{proof}

\begin{theorem}
	The Kleene lattice $\mathbf K$ depicted in Figure~5
	is not representable.
\end{theorem}

\begin{proof}
	Assume $\mathbf K$ to be representable. Then there exists some finite distributive lattice $\mathbf L=(L,\vee,\wedge)$ and some non-empty subset $S$ of $L$ with $\big(P_S(\mathbf L),\sqcup,\sqcap\big)\cong\mathbf K$. Because of Lemma~\ref{lem2}, without loss of generality we can assume $S=\{a\}$ or $S=\{a,b\}$ with $a<b$. Because of (i) of Lemma~\ref{lem3} we have $S=\{a\}$ with some $a\in L$. From (iii) of Lemma~\ref{lem4} we conclude $|L|\leq(|P_a(\mathbf L)|+1)/2=(|K|+1)/2=(9+1)/2=5$. From 
	 Corollary~\ref{lemmacbc} and Lemma~\ref{Lemma3c} it is easy to see that $\mathbf L$ can neither 
	be a chain nor the four-element Boolean algebra.  
	
	If $\mathbf L$ is of the form visualized in Figure~6 and $a$ is the smallest element of $\mathbf L$ then $\big(P_a(\mathbf L),\sqcup,\sqcap\big)$ has the form depicted in Figure~7. If $a$ is another element of $\mathbf L$ then $|P_a(\mathbf L)|\neq9$.

	If $\mathbf L$ is the dual of the lattice visualized in Figure~6 the same 
	reasons apply by Lemma \ref{Lemma1}. 
\end{proof}

\vspace*{4mm}

\begin{center}
\setlength{\unitlength}{7mm}

\begin{tabular}{c c c c c}
\begin{picture}(2,6)
\put(1,0){\circle*{.3}}
\put(1,1){\circle*{.3}}
\put(0,2){\circle*{.3}}
\put(2,2){\circle*{.3}}
\put(1,3){\circle*{.3}}
\put(0,4){\circle*{.3}}
\put(2,4){\circle*{.3}}
\put(1,5){\circle*{.3}}
\put(1,6){\circle*{.3}}
\put(1,1){\line(-1,1)1}
\put(1,1){\line(0,-1)1}
\put(1,1){\line(1,1)1}
\put(0,2){\line(1,1)2}
\put(2,2){\line(-1,1)2}
\put(1,5){\line(-1,-1)1}
\put(1,5){\line(0,1)1}
\put(1,5){\line(1,-1)1}
\put(.2,-1.4){{\rm Fig.~5}}
\end{picture}&\phantom{xxxxxx}&
\begin{picture}(2,3)
\put(1,0){\circle*{.3}}
\put(1,1){\circle*{.3}}
\put(0,2){\circle*{.3}}
\put(2,2){\circle*{.3}}
\put(1,3){\circle*{.3}}
\put(1,1){\line(-1,1)1}
\put(1,1){\line(1,1)1}
\put(1,1){\line(0,-1)1}
\put(1,3){\line(-1,-1)1}
\put(1,3){\line(1,-1)1}
\put(.2,-1.4){{\rm Fig.~6}}
\end{picture}&\phantom{xxxxxx}&
\begin{picture}(2,6)
\put(1,0){\circle*{.3}}
\put(0,1){\circle*{.3}}
\put(2,1){\circle*{.3}}
\put(1,2){\circle*{.3}}
\put(1,3){\circle*{.3}}
\put(1,4){\circle*{.3}}
\put(0,5){\circle*{.3}}
\put(2,5){\circle*{.3}}
\put(1,6){\circle*{.3}}
\put(1,0){\line(-1,1)1}
\put(1,0){\line(1,1)1}
\put(1,2){\line(-1,-1)1}
\put(1,2){\line(0,1)2}
\put(1,2){\line(1,-1)1}
\put(1,4){\line(-1,1)1}
\put(1,4){\line(1,1)1}
\put(1,6){\line(-1,-1)1}
\put(1,6){\line(1,-1)1}
\put(.2,-1.4){{\rm Fig.~7}}
\end{picture}
\end{tabular}
\end{center}

\vspace*{10mm}

\section*{Acknowledgment}

Support of the research of all authors by 
the Austrian Science Fund (FWF), project I~4579-N, and the Czech Science Foundation (GA\v CR), project 20-09869L, entitled ``The many facets of orthomodularity'' is gratefully acknowledged.  The first two authors gratefully acknowledge the support by \"OAD, project CZ~02/2019, entitled ``Function algebras and ordered structures related to logic and data fusion'', and, the first author  gratefully acknowledges the support by IGA, project P\v rF~2021~030.

\end{document}